\documentclass[a4paper, reqno, 12pt]{amsart}
\usepackage{hyperref}
\usepackage[usenames,dvipsnames]{color}
\usepackage{amsthm,amsfonts,amssymb,amsmath,amsxtra, array}
\usepackage[all]{xy}
\usepackage{xr-hyper}
\usepackage{verbatim}
\usepackage{tikz-cd, epic}
\usepackage[margin=1.25in]{geometry}
\usepackage{mathrsfs}
\usepackage{bbm}
\usepackage{float}
\restylefloat{table}
\usepackage{MnSymbol}

\RequirePackage{xspace}
\RequirePackage{etoolbox}
\RequirePackage{varwidth}
\RequirePackage[shortlabels]{enumitem}
\RequirePackage{mathtools}
\RequirePackage{longtable}

\def\<{\langle}
\def\>{\rangle}

\newcommand{{\BG}}{\ensuremath{\mathbb {G}}\xspace}

\newcommand{{\BK}}{\ensuremath{\mathbb {K}}\xspace}

\newtheorem{theorem}{Theorem}

\theoremstyle{definition}

\newtheorem{remark}[theorem]{Remark}

\newtheorem*{corollary to thm}{Corollary to Theorem \ref{dhkmthm}}

\setitemize[0]{leftmargin=*,itemsep=\the\smallskipamount}
\setenumerate[0]{leftmargin=*,itemsep=\the\smallskipamount}

\renewcommand{\to}{%
   \ifbool{@display}{\longrightarrow}{\rightarrow}%
   }
\let\shortmapsto\mapsto
\renewcommand{\mapsto}{%
   \ifbool{@display}{\longmapsto}{\shortmapsto}%
   }
\newlength{\olen}
\newlength{\ulen}
\newlength{\xlen}
\newcommand{\xra}[2][]{%
   \ifbool{@display}%
      {\settowidth{\olen}{$\overset{#2}{\longrightarrow}$}%
       \settowidth{\ulen}{$\underset{#1}{\longrightarrow}$}%
       \settowidth{\xlen}{$\xrightarrow[#1]{#2}$}%
       \ifdimgreater{\olen}{\xlen}%
          {\underset{#1}{\overset{#2}{\longrightarrow}}}%
          {\ifdimgreater{\ulen}{\xlen}%
             {\underset{#1}{\overset{#2}{\longrightarrow}}}
             {\xrightarrow[#1]{#2}}}}%
      {\xrightarrow[#1]{#2}}
   }
\makeatother
\newcommand{\xyra}[2][]{%
   \settowidth{\xlen}{$\xrightarrow[#1]{#2}$}%
   \ifbool{@display}%
      {\settowidth{\olen}{$\overset{#2}{\longrightarrow}$}%
       \settowidth{\ulen}{$\underset{#1}{\longrightarrow}$}%
       \ifdimgreater{\olen}{\xlen}%
          {\mathrel{\xymatrix@M=.12ex@C=3.2ex{\ar[r]^-{#2}_-{#1} &}}}%
          {\ifdimgreater{\ulen}{\xlen}%
             {\mathrel{\xymatrix@M=.12ex@C=3.2ex{\ar[r]^-{#2}_-{#1} &}}}
             {\mathrel{\xymatrix@M=.12ex@C=\the\xlen{\ar[r]^-{#2}_-{#1} &}}}}}%
      {\mathrel{\xymatrix@M=.12ex@C=\the\xlen{\ar[r]^-{#2}_-{#1} &}}}%
   }
\makeatletter
\newcommand{\xla}[2][]{%
   \ifbool{@display}%
      {\settowidth{\olen}{$\overset{#2}{\longleftarrow}$}%
       \settowidth{\ulen}{$\underset{#1}{\longleftarrow}$}%
       \settowidth{\xlen}{$\xleftarrow[#1]{#2}$}%
       \ifdimgreater{\olen}{\xlen}%
          {\underset{#1}{\overset{#2}{\longleftarrow}}}%
          {\ifdimgreater{\ulen}{\xlen}%
             {\underset{#1}{\overset{#2}{\longleftarrow}}}
             {\xleftarrow[#1]{#2}}}}%
      {\xleftarrow[#1]{#2}}
   }
\newcommand{\isoarrow}{%
   \ifbool{@display}{\overset{\sim}{\longrightarrow}}{\xrightarrow\sim}%
   }


\begin{document}

\title{A note on the admissibility of smooth simple $RG$-modules}
\author{Mihir Sheth}
\address{Department of Mathematics, Indian Institute of Science \\ Bangalore - 560012, India.}
\email{mihirsheth@iisc.ac.in}

\maketitle

\begin{abstract}
	Let $G$ be a $p$-adic reductive group and $R$ be a noetherian Jacobson $\mathbb{Z}[1/p]$-algebra. In this note, we show that every smooth irreducible $R$-linear representation of $G$ is admissible using the finiteness result of Dat, Helm, Kurinczuk and Moss for Hecke algebras over $R$.
\end{abstract}

\vspace{2mm}

Unless mentioned otherwise, all rings are commutative with unity. A $p$-adic reductive group is the group of rational points of a reductive group defined over a non-archimedean local field of residue characteristic $p>0$. Let $R$ be a ring and $G$ be a $p$-adic reductive group. Let $RG$ denote the group algebra. An $RG$-module $\pi$ is called \emph{smooth} if every $v\in\pi$ is fixed by some compact open subgroup in $G$, and \emph{admissible} if for every compact open subgroup $K\subseteq G$, $\pi^{K}$ is a finitely generated $R$-module. 

For a compact open subgroup $K\subseteq G$, let $H_{R}(G,K)$ denote the Hecke algebra of compactly supported $R$-valued $K$-biinvariant functions on $G$ equipped with the convolution product and $Z_{R}(G,K)$ denote its center. The Hecke algera $H_{R}(G,K)$ is an associative $R$-algebra with unity. The following theorem is the joint work \cite[Theorem 1.2]{dhkm24} of Dat, Helm, Kurinczuk and Moss (see also \cite{dhkm22}):

\begin{theorem}\label{dhkmthm}
For any noetherian $\mathbb{Z}[1/p]$-algebra $R$ and any compact open subgroup $K\subseteq G$, the Hecke algebra $H_{R}(G,K)$ is a finitely generated module over $Z_{R}(G,K)$ and $Z_{R}(G,K)$ is a finitely generated $R$-algebra.  
\end{theorem}
 
As an application of Theorem \ref{dhkmthm}, we prove the following result:

\begin{theorem}\label{mainthm}
If $R$ is a noetherian Jacobson $\mathbb{Z}[1/p]$-algebra, then any smooth simple $RG$-module $\pi$ is admissible.
\end{theorem}

Recall that a ring is called Jacobson if every prime ideal is the intersection of the
maximal ideals which contain it. The examples of noetherian Jacobson $\mathbb{Z}[1/p]$-algebras include all fields of characteristic not equal to $p$ as well as finitely generated algebras over such fields such as $\mathbb{F}_{l}[T_{1},\ldots,T_{n}]$ or finite rings such as $\mathbb{Z}/l^{m}\mathbb{Z}$ with $l\neq p$ for which Theorem \ref{mainthm} is a new result. When $R=\mathbb{C}$, Theorem \ref{mainthm} is a classical result in the representation theory of $p$-adic groups. The case when $R$ is any field of characteristic not equal to $p$ is given in \cite[Proposition 4.10]{vh}. We remark that Theorem \ref{mainthm} does not hold for representations over a field of characteristic $p$, see \cite{gls23}.


Theorem \ref{mainthm} follows from the following Corollary to Theorem \ref{dhkmthm}:

\begin{corollary to thm}\label{cor}
\emph{Let $R$ be a noetherian Jacobson $\mathbb{Z}[1/p]$-algebra, and $M$ be a simple (left) module over $H_{R}(G,K)$. Then $M$ is a finitely generated $R$-module.}
\end{corollary to thm}
\begin{proof}
To ease the notation, let us write $H=H_{R}(G,K)$ and $Z=Z_{R}(G,K)$. Choose a surjective map $H\twoheadrightarrow M$ of $R$-modules. Let $\mathfrak{m}$ be the kernel of the surjection $H\twoheadrightarrow M$ and $\mathfrak{m}_{Z}:=\mathfrak{m}\cap Z$. Note that $\mathfrak{m}_{Z}$ is a two-sided ideal of $Z$ because $Z$ is the center of $H$. We claim that $\frac{Z}{\mathfrak{m}_{Z}}$ is a field. Let $\bar{z}:=z+\mathfrak{m}_{Z}\in \frac{Z}{\mathfrak{m}_{Z}}$ be a non-zero element. Then $\bar{z}$ is also non-zero in $\frac{H}{\mathfrak{m}}$. So the left $H$-submodule $H\bar{z}$ of $\frac{H}{\mathfrak{m}}$ generated by $\bar{z}$ is equal to $\frac{H}{\mathfrak{m}}$ because $\frac{H}{\mathfrak{m}}$ is simple. Therefore, there exists $h\in H$ such that $\bar{h}\bar{z}=\bar{1}$ in $\frac{H}{\mathfrak{m}}$.

Consider the $\frac{Z}{\mathfrak{m}_{Z}}$-algebra $\frac{Z}{\mathfrak{m}_{Z}}[\bar{h}]$ generated by $\bar{h}$. It is commutative because $\frac{Z}{\mathfrak{m}_{Z}}$ is commutative. Moreover, as $\frac{H}{\mathfrak{m}}$ is a finitely generated $\frac{Z}{\mathfrak{m}_{Z}}$-module and $\frac{Z}{\mathfrak{m}_{Z}}$ is noetherian, $\frac{Z}{\mathfrak{m}_{Z}}[\bar{h}]$ is a finitely generated $\frac{Z}{\mathfrak{m}_{Z}}$-module. Hence, $\bar{h}$ is integral over $\frac{Z}{\mathfrak{m}_{Z}}$, i.e. \[\bar{h}^{n}+\bar{a}_{n-1}\bar{h}^{n-1}+\ldots+\bar{a}_{0}=0,\] for some $n\in\mathbb{N}$ and $\bar{a}_{n-1},\bar{a}_{n-2},\ldots,\bar{a}_{0}\in \frac{Z}{\mathfrak{m}_{Z}}$. Multiplying both sides of the above by $\bar{z}^{n-1}$ and using that $\frac{Z}{\mathfrak{m}_{Z}}$ commutes with $\bar{h}$, we obtain that \[\bar{h}+\bar{a}_{n-1}+\bar{a}_{n-2}\bar{z}+\ldots+\bar{a}_{0}\bar{z}^{n-1}=0.\] Hence $\bar{h}=-\left(\bar{a}_{n-1}+\bar{a}_{n-2}\bar{z}+\ldots+\bar{a}_{0}\bar{z}^{n-1}\right)\in\frac{Z}{\mathfrak{m}_{Z}}$.   

Now the field $\frac{Z}{\mathfrak{m}_{Z}}$ is a finitely generated $R$-algebra. One of the characterizations of Jacobson rings implies that $\frac{Z}{\mathfrak{m}_{Z}}$ is a finitely generated $R$-module \cite[Theorem 10]{eme}. Since $\frac{H}{\mathfrak{m}}$ is finite over $\frac{Z}{\mathfrak{m}_{Z}}$, we get that $\frac{H}{\mathfrak{m}}\cong M$ is also a finitely generated $R$-module.
\end{proof}

\begin{proof}[Proof of Theorem \ref{mainthm}]
Since $G$ has a fundamental system of neighborhoods of identity consisting of open pro-$p$ subgroups, it is enough to show that $\pi^{K}$ is a finitely generated $R$-module for $K\subseteq G$ an open pro-$p$ subgroup. Let $K\subseteq G$ be an open pro-$p$ subgroup such that $\pi^{K}\neq 0$. Since $\pi$ is simple and $p\in R^{\times}$, $\pi^{K}$ is a simple $H_{R}(G,K)$-module by \cite[I.6.3]{vig96}. Hence, $\pi^{K}$ is a finitely generated $R$-module by the Corollary to Theorem \ref{dhkmthm}. 
\end{proof}

\begin{remark}\label{converse of hecke module statement}
	The requirement for $R$ to be Jacobson in Corollary to Theorem \ref{dhkmthm} is necessary. Indeed, if $R$ is a commutative ring and if all simple modules over $H_{R}(G,K)$ are finitely generated $R$-modules for all $p$-adic reductive groups $G$ and compact open subgroups $K$, then $R$ is Jacobson. The following proof of this converse statement was communicated to us by M.-F. Vign\'{e}ras: By Satake \cite[\S 8]{sat63}, if $G$ is a classical simple group with trivial center and $K\subseteq G$ a natural maximal compact subgroup, then $H_{R}(G,K)$ is a polynomial ring over $R$ in $m$ variables where $m$ is the rank of a maximal split torus in $G$. Thus, a finitely generated $R$-algebra $A$ is a quotient of some $H_{R}(G,K)$. If $A$ is a field, then $A$ is a simple module over $H_{R}(G,K)$, and hence a finitely generated $R$-module by assumption. This means that $R$ is Jacobson.      
\end{remark}	

\begin{remark}
Let $R=\mathbb{Z}_{l}$ with $l\neq p$ and $G=\mathrm{GL}_{2}(\mathbb{Q}_{p})$. As $R$ is not Jacobson, Remark \ref{converse of hecke module statement} suggests that $G$ admits a smooth irreducible $\mathbb{Z}_{l}$-representation that is not admissible. Indeed, let $K=\mathrm{GL}_{2}(\mathbb{Z}_{p})$. By \cite[Proposition 2.1]{gk14}, $H=H_{R}(G,K)\cong R[T_{0},T_{0}^{-1},T_{1}]$. One can make $M:=\mathbb{Q}_{l}$ into a simple $H$-module by defining the action via the surjective map $H\twoheadrightarrow\mathbb{Q}_{l}$ which takes $T_{0}$ to $1$ and $T_{1}$ to $l^{-1}$. However, note that $M$ is not a finitely generated $R$-module. By choosing a prime $l$ so that the pro-order of $K$ is invertible in $R$, there exists a smooth simple $RG$-module $\pi$ such that $\pi^{K}\cong M$ as $H$-modules \cite[I.4.4 and I.6.3]{vig96}. Since $\pi^{K}$ is not a finite $R$-module, $\pi$ is non-admissible.    
\end{remark}

\noindent {\it Acknowledgments}: The author thanks Radhika Ganapathy and M.-F. Vign\'{e}ras for many helpful discussions.

\end{document}